\documentclass{amsart}

\usepackage{amssymb}
\usepackage{amsmath}
\usepackage{amsthm}
\usepackage{amsbsy}
\usepackage{graphicx}
\usepackage{bm}
\usepackage{setspace}\spacing{1.67}

\theoremstyle{plain}
\newtheorem{theorem}{Theorem}[section]
\newtheorem{lemma}[theorem]{Lemma}

\newtheorem{proposition}[theorem]{Proposition}
\newtheorem*{quotthm}{Theorem}

\theoremstyle{definition}

\theoremstyle{remark}

\newtheorem*{acknowledgements}{Acknowledgements}

\newcommand{\Z}{\mathbb{Z}}

\newcommand{\C}{\mathbb{C}}

\newcommand{\co}{\colon\thinspace}
\newcommand{\del}{\partial}

\begin{document} 

\title{Degree-one maps, surgery and four-manifolds}

\author{Siddhartha Gadgil}

\address{	Department of Mathematics\\
		Indian Institute of Science\\
		Bangalore 560003, India}

\email{gadgil@math.iisc.ernet.in}

\date{\today}



\begin{abstract}
We give a description of degree-one maps between closed,
oriented $3$-manifolds in terms of surgery. Namely, we show that there
is a degree-one map from a closed, oriented $3$-manifold $M$ to a
closed, oriented $3$-manifold $N$ if and only if $M$ can be obtained
from $N$ by surgery about a link in $N$ each of whose components is an
unknot.

We use this to interpret the existence of degree-one maps between
closed $3$-manifolds in terms of smooth $4$-manifolds. More precisely,
we show that there is a degree-one map from $M$ to $N$ if and only if
there is a smooth embedding of $M$ in $W=(N\times I)\#_n \overline{\C
P^2}\#_m {\C P^2}$, for some $m\geq 0$, $n\geq 0$ which separates the
boundary components of $W$. This is motivated by the relation to
topological field theories, in particular the invariants of Ozsvath
and Szabo.
\end{abstract}

\maketitle

\section{Introduction}

We assume that all manifolds are connected and that all $3$-manifolds
are smooth. For closed, oriented $3$-manifolds $M$ and $N$, we say
that $M$ \emph{dominates} $N$ (or $M$ 1-dominates $N$) if there is a
degree-one map from $M$ to $N$. This gives a transitive relation on
closed, oriented $3$-manifolds which has been extensively studied by
several authors (for instance, see~\cite{BW}, \cite{HWZ}, \cite{ReW},
\cite{Ro}, \cite{RW}, \cite{So}, \cite{WZ}). Note that every manifold
dominates $S^3$ and that if $M$ dominates $N$ then there is a
surjection from $\pi_1(M)$ to $\pi_1(N)$.

In this paper, we characterise dominance in terms of Dehn surgery. We
use this to interpret dominance in terms of smooth $4$-manifolds. The
latter is motivated by the relation to topological field theories, in
particular the invariants of Ozsvath and Szabo~\cite{OZ}\cite{OZ1}.

Suppose $N$ is a closed $3$-manifold and $M$ is obtained from $N$ by
surgery about a link in $N$ all of whose components are homotopically
trivial, then it is easy to see that there is a degree-one map from
$M$ to $N$. Our first result is the converse, namely that if there is
a degree-one map from $M$ to $N$, then $M$ can be obtained from $N$ by
surgery about a link $L\subset N$ all of whose components are
homotopically trivial. In fact we can find $L$ each of whose components
is an unknot.

\begin{theorem}\label{surger}
For closed oriented $3$-manifolds $M$ and $N$, there is a degree-one
map from $M$ to $N$ if and only if $M$ can be obtained from $N$ by
surgery about a link in $N$ each of whose components is an unknot in
$N$.
\end{theorem}

We next interpret dominance of $3$-manifolds in terms of
$4$-manifolds. Observe that a partial ordering on closed orientable
$3$-manifolds can be defined by saying that $M$ \emph{strongly
dominates} $N$ if there is a smooth embedding $i\co M\to N\times
(0,1)\subset N\times [0,1]$ so that $i(M)$ separates the two boundary
components $N\times \{0\}$ and $N\times \{1\}$, with the
\emph{appropriate orientation}. Observe that if $M$ strongly dominates
$N$, the composition $\pi\circ i$ of the embedding $i$ with the
projection $\pi\co N\times [0,1]\to N$ has degree $\pm 1$. We say that
the embedding has the \emph{appropriate orientation} if the degree of
this map is one. Such a definition is related to the theory of
\emph{imitations} introduced by Kawauchi~\cite{Ka}.

This definition is motivated by the relation to ($(3+1)$-dimensional)
\emph{topological field theories}, in particular the invariants of
Ozsvath and Szabo (however, our methods do not apply to the
Ozsvath-Szabo theory because of the dependence on $Spin^c$
structures). Recall that a degree-one map $f\co M\to N$ induces a
surjection $f_*$ on the level of fundamental groups. Hence if $\pi_1(N)$ is
non-trivial so is $\pi_1(M)$. Further $f^*\co H^*(N)\to H^*(M)$ is an
injection, which shows that if $H^k(N)\neq 0$ then $H^k(M)\neq 0$.

We see that an analogous result holds for any topological field
theory, with dominance replaced by strong dominance. Recall that a
$(3+1)$-dimensional topological field theory associates to each
closed, oriented $3$-manifold $M$ a vector space $V(M)$ and to each
cobordism $W$ from $M$ to another closed, oriented $3$-manifold $N$ a
linear transformation $T(W)\co V(M)\to V(N)$. Further this satisfies
functorial properties, namely a product cobordism induces the identity
map and if $W_1$ is a cobordism from $M_1$ to $M_2$ and $W_2$ is a
cobordism from $M_2$ to $M_3$ then for the cobordism $W_1\coprod_{M_2}
W_2$ from $M_1$ to $M_3$, $T(W_1\coprod_{M_2} W_2)=T(W_2)\circ
T(W_1)$.

Suppose $M$ strongly dominates $N$, then splitting $N\times [0,1]$
along the given embedding of $M$ gives two cobordisms, $W_1$ from $N$
to $M$ and $W_2$ from $M$ to $N$. The composition of these is the
product cobordism $N\times [0,1]$, which induces the identity map on
$T(N)$. It follows that the identity map on $V(N)$ factors through
$V(M)$, and in particular $V(N)\neq 0$ implies that $V(M)\neq 0$. This
is the analogue of the corresponding results for $\pi_1$ and $H^*$
with respect to degree-one maps. Thus \emph{strong dominance} plays
the same role in the \emph{bordism category} as dominance in the
homotopy category.

We shall see (in Proposition~\ref{Poinc}) that the relation of strong
dominance is stronger than dominance.  We show, however, that
dominance is equivalent to a relation obtained using $4$-manifolds
similar to the above one except that we allow `positive and negative
blow-ups'.

\begin{theorem}\label{embed}
For closed orientable $3$-manifolds $M$ and $N$, there is a degree-one
map from $M$ to $N$ if and only if there is a smooth embedding of $M$
in $int(W)$, $W=(N\times I)\#_n \overline{\C P^2}\#_m {\C P^2}$ for
some $m>0$, $n>0$ which separates the boundary components of $W$, with
the embedding having the appropriate orientation.
\end{theorem}

There is a relation in between dominance and strong dominance which is
of interest. Namely, we say that $M$ \emph{negatively dominates} $N$
if there is an embedding of $M$ into $W=(N\times I)\#_n \overline{\C
P^2}$, for some $n\geq 0$, which separates the two boundary
components. This is of interest because the Ozsvath-Szabo invariants
(as also the Seiberg-Witten invariants) behave well under blowing up.

We shall see (in Proposition~\ref{Poinc}) that the Poincar\'e homology
sphere does not even negatively dominate $S^3$. We shall study
negative dominance elsewhere.

\begin{acknowledgements}
I thank the referee for many helpful comments.
\end{acknowledgements}

\section{Degree-one maps and Surgery}

In this section, we give a proof of Theorem~\ref{surger}. Suppose $M$
is obtained from $N$ by surgery about a link $L\subset N$ each of
whose components $K_i$ is homotopically trivial. Then it is shown
in~\cite{BW}) that there is a degree-one map from $M$ to $N$.

The converse is based on the following theorem of Haken~\cite{Ha} and
Waldhausen~\cite{Wa} (see also~\cite{RW}).

\begin{quotthm}[Haken-Waldhausen]
Let $f\co M\to N$ be a degree-one map and let $N=H_1\cup H_2$ be a
Heegaard decomposition of $N$ with $H_1$ and $H_2$ handlebodies. Then
$f$ is homotopic to a map $g$ such that
$g|_{g^{-1}(H_1)}\co g^{-1}(H_1)\to H_1$ is a homeomorphism.
\end{quotthm}

Such a map is called a $1$-pinch. Thus if $M$ dominates $N$, there is
a $1$-pinch $g\co M\to N$.

\begin{proof}[Proof of Theorem~\ref{surger}]
Assume $M$ dominates $N$ and let $g$, $H_1$ and $H_2$ be as above.
Consider a collection of properly embedded discs $D_i$, $1\leq i\leq
n$, in $H_2$ such that on splitting $H_2$ along $D_i$, we get a
$3$-ball $B$. We can assume that $g$ is transversal to $D_i$ for all
$i$, $1\leq i\leq n$. Let $F_i=g^{-1}(D_i)$ and let
$P=g^{-1}(H_2)$. Note that $g|_{\del H_2}$ is a homeomorphism and
hence $F_i$ consists of a compact surface with a single boundary
component and a (possibly empty) collection of closed surfaces. First,
note that by performing a homotopy of $g$ we can assume that each
$F_i$ is connected. This follows (as the induced map on $\pi_1$ is a
surjection) by using standard techniques using binding ties as in
Stallings' proof of the Knesser conjecture (see for example the proof
of Knesser's conjecture in~\cite{He}). Hence $F_i$, $1\leq i\leq n$,
is a compact surface with a single boundary component.

We first consider the special case when each $F_i$ is a disc. 

\begin{lemma}
Suppose $F_i=g^{-1}(D_i)$ is a disc for each $i$, $1\leq i\leq
n$. Then $P$ is obtained from the handlebody $H_2$ by surgery about a
link, each of whose components is an unknot.
\end{lemma}
\begin{proof}
After a homotopy of $g$, we can assume that $F_i$ maps
homeomorphically onto $D_i$ for $1\leq i\leq n$. On splitting $P$
along the properly embedded discs $F_i$, $1\leq i\leq n$, we get a
manifold $\hat P$ with boundary a $2$-sphere. By the theorem of
Lickorish and Wallace~\cite{Li}\cite{Wal}, this can be obtained from
$B$ by surgery about a link $L$ in $B$, with each component of $L$ an
unknot. Thus $P$ is obtained from $H_2$ by surgery about a link, each
component of which is an unknot.
\end{proof}

It follows that, in this special case, $M$ is obtained from $N$ by
surgery about a link each component of which is an unknot. We now turn
to the general case. 

In the general case, we shall perform surgery on $M$ to obtain a
manifold $M'$ and a degree-one map $g':M'\to N$ which is as in the
special case. Hence $M'$ is obtained from $N$ by surgery about a link,
each component of which is an unknot. Further $M$ is obtained from
$M'$ by surgery, so we get a link in $N$ so that surgery about this
link gives $M$. We shall show that each component of this link is
homotopically trivial in $N$. From this, we deduce that we can obtain
$M$ from $N$ by surgery about a link, each component of which is an
unknot.

First, we construct $M'$ and $g'$.

\begin{lemma}\label{spcl}
There is a framed link $L'\subset P\subset M$ such that, if $M'$ is
the result of surgery of $M$ about $L'$, there is a degree-one map
$g':M'\to N$, which co-incides with $g$ outside a neighbourhood of
$L'$, so that $g'^{-1}(D_i)$ is a disc for each $i$. Furthermore, if
$P'\subset M'$ is the result of surgery of $P$ about $L'$, then
$g'(P')\subset H_2$.
\end{lemma}
\begin{proof}
For each $i$, $1\leq i\leq n$, consider a collection $L_i$ of
disjoint, embedded simple closed curves on $F_i$ which do not separate
$F_i$ and are maximal with respect to this property. Then $L'=\cup_i
L_i$ is a link in $M$. We consider a corresponding framed link (also
denoted $L'$), with the framing of a component of $L_i$ given by the
normal to $F_i$. Let $M'$ be the manifold obtained from $M$ by surgery
about the framed link $L'$.

We shall see that the map $g$ induces a degree-one map $g'$ from $M'$
to $N$ with $g'^{-1}(D_i)$ obtained from $F_i$ by compressing along
the components of $L_i$. By the choice of $L_i$ it follows that
$g'^{-1}(D_i)$ is a disc for all $i$, $1\leq i\leq n$.

Let the components of $L_i$ be $C^i_j$ and let $T^i_j$ denote a
regular neighbourhood of $C^i_j$. On the complement of $\bigcup_{i,j}
int(T^i_j)$, we let $g'=g$. After surgery, each $T^i_j$ is replaced by
a solid torus $X=X^i_j$ with the same boundary as $T^i_j$. Now,
$F_i\cap \del X=F_i\cap \del T^i_j$ is the union of two parallel
curves $\mu_1$ and $\mu_2$ (on $\del X$). Furthermore, by the choice
of the surgery slope, $\mu_1$ and $\mu_2$ are meridians in $X$ ,i.e.,
they bound properly embedded discs $E_1$ and $E_2$ in $X$. As
$g(\mu_i)\subset D_i$ and $\del E_i=\mu_i$, the map $g'$ extends to
$E_i$ with $g'(E_i)\subset D_i$. Using transversality of $g$ to $D_i$,
we see that we can extend $g'$ to a regular neighbourhood $E_i\times
[-1,1]$ of $E_i$ with $g'(E_i\times [-1,1]-E_i)\cap
D_i=\phi$. Finally, $X-(E_1\times (-1,1))-(E_2\times (-1,1))$ is the
union of two balls $B_1$ and $B_2$, each of whose boundaries consists
of two discs (one a component of $E_1\times \{-1,1\}$ and one a
component of $E_2\times \{-1,1\}$) and an annulus in $X$ disjoint from
$\mu_1$ and $\mu_2$.  The function $g'$ has been defined on $\del B_i$
for $i=1,2$ and, by construction and using the fact that $T^i_j$ is
disjoint from $F_l$ for $l\neq i$, the image of $g'(\del B_k)$ is
contained in $B$ for $k=1,2$. Thus $g'$ extends to a map on $X$ with
$g'(B_k)\subset B$, $k=1,2$. It follows that
$(g'|_X)^{-1}(D_i)=E_1\cup E_2$. Making this construction for each
$X^i_j$, we get a map $g'$ as claimed.
\end{proof}

Now, by applying Lemma~\ref{spcl} to $g'$, we see that $M'$ can be
obtained from $N$ by surgery about a link $L_0\subset H_2\subset N$,
each of whose components is homotopically trivial in $H_2\subset
N$. Surgery of $H_2$ about $L_0$ gives the manifold (with boundary)
$P'$. We now perform surgeries about knots $\gamma^i_j\subset P'\subset
M'$ so that the surgery about $\gamma^i_j$ cancels the surgery about
$C^i_j\subset L'$. Thus on performing such surgeries we obtain $M$. As
$P'$ is obtained from $H_2$ by surgery and the knots
$\gamma^i_j\subset K'$ can be perturbed to be disjoint from the locus
of the surgery, they can be regarded as knots in $H_2\subset N$. Thus
the union of the knots $\gamma^i_j$, with framing corresponding to the
canceling surgeries, is a framed link $L_1\subset H_2\subset N$.

Thus $M$ is obtained from $N$ be surgery about the framed link
$L=L_0\cup L_1$, with each component of $L_0$ homotopically
trivial. We next show that the knots $\gamma^i_j$, regarded as curves
in $H_2\subset N$, are homotopically trivial.

\begin{lemma}
The knots $\gamma^i_j$, regarded as curves in $H_2\subset N$, are
homotopically trivial.
\end{lemma}
\begin{proof}
Recall that $\gamma^i_j$ is the knot corresponding to the surgery
canceling the surgery about $C^i_j$. Hence it is obtained by pushing
off a meridian of $T^i_j$. Thus, $\gamma^i_j$ intersects $F_i$
transversally in two points, with opposite signs of intersection, and
$\gamma^i_j$ is disjoint from $F_k$ for $k\neq i$.

Note that $\pi_1(H_2)$ is a free group with generators $\alpha_i$,
$1\leq i\leq n$, corresponding to the discs $D_i$. Further, if
$\gamma$ is a curve transversal to the discs $D_i$, $1\leq i\leq n$,
then (up to conjugacy) the word represented by $\gamma$ is determined
by the intersection points with the discs $D_i$. Namely, if the points
of $\gamma\cap(\cup_i D_i)$, in cyclic order around $\gamma$, are
contained in $D_{i_1}$,\dots $D_{i_k}$ with signs of intersection
$\epsilon_j=\pm 1$, then
$\gamma=\alpha_{i_1}^{\epsilon_1}\dots\alpha_{i_1}^{\epsilon_k}$ up to
conjugacy.

As $\gamma^i_j$ intersects $F_i$ (hence $D_i$) transversally in two
points, with opposite signs of intersection, and $\gamma^i_j$ is
disjoint from $F_k$ (hence $D_k$) for $k\neq i$, it follows that
$\gamma^i_j$ represents the trivial word in $\pi_1(H_2)$, and hence is
homotopically trivial in $N$.
\end{proof}

Thus, $M$ is obtained from $N$ by surgery about a link, each component
of which is homotopically trivial. We shall deduce from this that we
can choose the link so that each component is an unknot.

\begin{lemma}
Suppose $M$ is obtained from $N$ by surgery about a link $L$, each
component of which is homotopically trivial. Then $M$ is obtained from
$N$ by surgery about a link $L'$, each component of which is an
unknot.
\end{lemma}
\begin{proof}
As each component of $L$ is homotopically trivial, there is a sequence
of crossing changes so that on performing these crossing changes we
obtain a link all of whose components are unknots. Observe that each
crossing change of a knot $\kappa$ is locally of a standard
form. Namely, there is a ball $B\subset M$ which intersects $\kappa$
in a pair of arcs $c_1$ and $c_2$, and the crossing change corresponds
to a crossing of these arcs to give new arcs $c'_1$ and $c'_2$ with
the same endpoints as $c_1$ and $c_2$.

Further, if $K_i$ is an unknot in $B$ unlinked from the arcs $c_i$
with framing $\pm 1$, then on performing the Kirby moves of sliding
$c_1$ and $c_2$ over $K_i$, with opposite orientations, we get the
knot obtained by crossing $c_1$ and $c_2$. In this manner we can
obtain both positive and negative crossing changes. 

Replacing $L$ by its union with unknots and performing the Kirby moves
as above does not change the resulting manifold. Thus, we can replace
$L$ by a framed link in $N$, each of whose components is an unknot, so
that the result of surgery about the link is $M$.
\end{proof}

This completes the proof of Theorem~\ref{surger}.
\end{proof}

\section{Surgery and $4$-manifolds}

We now characterise dominance in terms of $4$-manifolds.

\begin{proof}[Proof of Theorem~\ref{embed}]

Suppose $M$ embeds in $W$ as in the hypothesis. Then $W-M$ has two
components with closures $K_1$ and $K_2$ so that
$\del[K_1]=[M]-[N\times \{0\}]$. Hence $[M]$ is homologous to
$[N\times \{0\}]$. Now by identifying all the points in each $\C P^2$
and $\overline{\C P^2}$ in $W=(M\times I)\#_n \overline{\C P^2}\#_m
{\C P^2}$ to a single point, we get a blow-down map $\pi\co W\to N\times
[0,1]$. By composing with the projection, we get a map $p\co W\to N$ with
$p\co N\times \{0\}\to N$ being the identity map. This restricts to a map
$p\co M\to N$. As $[M]$ is homologous to $[N\times \{0\}]$,
$p_*([M])=[N]$, i.e., $M$ has degree one.

Conversely, assume $M$ and $N$ are as in the hypothesis. By
Theorem~\ref{surger}, $M$ can be obtained from $N$ by surgery about a
framed link $L$, all of whose components are unknots in $N$. Hence $L$
can be obtained from an unlink $L_0\subset N$ by a sequence of (say
$p$) crossings.

Let $K_1$,\dots $K_n$ a collection of unknots in $N$, with $n\geq p$
to be specified later, so that $L_0\cup\{ K_1,\dots K_n\}$ forms an
unlink. Let $W$ be obtained by attaching a $2$-handle with framing
$\pm 1$ (with signs to be chosen later) to $N\times [0,1/2]$ along
each of $K_0$, $K_1$,\dots $K_n$. Note that $W=(N\times
[0,1])\#_{k}\overline{\C P^2}\#_l {\C P^2}$ for some $k$ and $l$.

We shall construct a different Kirby diagram for $W$. Corresponding to
the $p$ crossings of $L_0$ required to make it isotopic to $L$ we can
find disjoint balls $B_i$, $1\leq i\leq p$, in which the crossing is
made.  By an isotopy, we can assume that for $1\leq i\leq p$, $K_i$ is
contained in $B_i$. Performing the Kirby moves corresponding to the
crossing changes in each of these $B_i$, we get a Kirby diagram for
$W$ with a sublink isotopic to $L$. Furthermore, by performing the
Kirby move of sliding over the unknots $K_{p+1}$,\dots, $K_n$ (with
framing $\pm 1$) we can ensure this sublink is isotopic to $L$ as a
framed link (as such a Kirby move changes the framing by $\pm 1$
without changing the link). Consider the corresponding Morse function
for $W$ with the $2$-handles corresponding to components of $L$
attached first.  The level set on attaching $L$ is the result of
surgery about $L$. But this is $M$, and hence we get an embedding of
$M$ separating the boundary components of $W$.
\end{proof}

We next see that strong dominance is not the same as dominance. As is
well known, any $3$-manifold dominates the $3$-sphere.  However, we
see that $S^3$ is not a minimal element with respect to strong
dominance or even negative dominance. This result has also been
observed by Ding~\cite{Di}.

\begin{proposition}\label{Poinc}
For $n\geq 0$, there is no embedding of the Poincare homology sphere
in $(S^3\times I)\#_n \overline{\C P^2}$ which separates the boundary
components.
\end{proposition}
\begin{proof}
Note that the Poincar\'e homology sphere can be obtained from $S^3$ by
surgery on the $E_8$ link, and hence, with one of its orientations,
bounds a $4$-manifold $W$ with positive definite intersection
pairing. Denote the Poincar\'e homology sphere with this orientation
as $M$.

Suppose, for some $n\geq 0$, there is an embedding of the Poincare
homology sphere in $(S^3\times I)\#_n \overline{\C P^2}$ which
separates the boundary components. Then by capping off the boundary
components $S^3\times \{0\}$ and $S^3\times\{1\}$, we get an embedding
of the Poincare homology sphere in $\#_n \overline{\C P^2}$. Splitting
along the embedding and using the Mayer-Vietoris sequence, we get
$4$-manifolds $W_1$ and $W_2$ bounding $M$ and $-M$ with positive
definite intersection forms. If $\del W_1=-M$, then $Y=W\coprod_M W_1$
is a smooth $4$-manifold with $H_1(Y,\Z)=0$ and the intersection form
on $H_2(Y,\Z)$ is positive definite but not diagonalisable, contradicting
Donaldson's theorem~\cite{Do}.

\end{proof}

 \bibliographystyle{amsplain}

\end{document}